\documentclass[11pt]{article}
\def\thetitle{Colour-biased Hamilton cycles in random graphs}

\usepackage{graphicx}

\usepackage{amsmath,amssymb}

\usepackage[usenames,dvipsnames,svgnames,table]{xcolor}
\definecolor{CombinatoricaAqua}{HTML}{00698C}
\definecolor{CombinatoricaBlue}{HTML}{3A3293}
\definecolor{CombinatoricaBrown}{HTML}{66220C}
\definecolor{CombinatoricaRed}{HTML}{DF2A27}
\definecolor{HarvardCrimson}{rgb}{0.6471, 0.1098, 0.1882}

\makeatletter
\let\reftagform@=\tagform@	
\def\tagform@#1{\maketag@@@
	{(\ignorespaces\textcolor{CombinatoricaBrown}{#1}\unskip\@@italiccorr)}}
\renewcommand{\eqref}[1]{\textup{\reftagform@{\ref{#1}}}}
\makeatother

\usepackage[backref=page]{hyperref}
\hypersetup{
	unicode,
	pdfencoding=auto,
	pdfauthor={Peleg Michaeli},
	pdftitle={\thetitle},
	pdfsubject={},
	pdfkeywords={},
	colorlinks=true,
	citecolor=CombinatoricaBlue,
	linkcolor=CombinatoricaAqua,
	anchorcolor=CombinatoricaBrown,
	urlcolor=HarvardCrimson}

\usepackage{amsthm}
\usepackage{thmtools}
\usepackage{bbm}
\usepackage{enumitem}
\usepackage[capitalize]{cleveref}
\Crefname{fact}{Fact}{Facts}
\Crefname{claim}{Claim}{Claims}

\declaretheoremstyle[
spaceabove=\topsep, spacebelow=\topsep,
headfont=\color{CombinatoricaBrown}\normalfont\bfseries,
bodyfont=\itshape,
]{thm}
\declaretheoremstyle[
spaceabove=\topsep, spacebelow=\topsep,
headfont=\color{CombinatoricaBrown}\normalfont\bfseries,
bodyfont=\normalfont,
]{dfn}
\declaretheoremstyle[
spaceabove=0.5\topsep, spacebelow=0.5\topsep,
headfont=\color{CombinatoricaBrown}\normalfont\bfseries,
bodyfont=\normalfont,
]{rmk}

\declaretheorem[style=thm,parent=section]{theorem}
\declaretheorem[style=thm,sibling=theorem]{lemma}
\declaretheorem[style=thm,sibling=theorem]{corollary}
\declaretheorem[style=thm,sibling=theorem]{claim}

\declaretheorem[style=definition,numbered=no]{acknowledgements}

\usepackage[nobysame,msc-links]{amsrefs}

\BibSpecAlias{misc}{webpage}
\BibSpec{book}{%
	+{}  {\PrintPrimary}                {transition}
	+{,} { \textbf}                     {title} % was \textit
	+{.} { }                            {part}
	+{:} { \textit}                     {subtitle}
	+{,} { \PrintEdition}               {edition}
	+{}  { \PrintEditorsB}              {editor}
	+{,} { \PrintTranslatorsC}          {translator}
	+{,} { \PrintContributions}         {contribution}
	+{,} { }                            {series}
	+{,} { \voltext}                    {volume}
	+{,} { }                            {publisher}
	+{,} { }                            {organization}
	+{,} { }                            {address}
	+{,} { \PrintDateB}                 {date}
	+{,} { }                            {status}
	+{}  { \parenthesize}               {language}
	+{}  { \PrintTranslation}           {translation}
	+{;} { \PrintReprint}               {reprint}
	+{.} { }                            {note}
	+{.} {}                             {transition}
	+{}  {\SentenceSpace \PrintReviews} {review}
}

\makeatletter
\renewcommand{\PrintNames@a}[4]{%
	\PrintSeries{\name}
	{#1}
	{}{ and \set@othername}
	{,}{ \set@othername}
	{}{ and \set@othername}
	{#2}{#4}{#3}%
}
\makeatother

\makeatletter
\def\mathcolor#1#{\@mathcolor{#1}}
\def\@mathcolor#1#2#3{%
	\protect\leavevmode
	\begingroup
	\color#1{#2}#3%
	\endgroup
}
\makeatother
\definecolor{Red}{rgb}{0.618,0,0}
\definecolor{Blue}{rgb}{0,0,1}
\definecolor{Green}{rgb}{0,0.298,0}

\usepackage{sectsty}
\sectionfont{\color{CombinatoricaBrown}}
\subsectionfont{\color{CombinatoricaBrown}}
\subsubsectionfont{\color{CombinatoricaBrown}}

\usepackage{soul}
\soulregister\ref7

\usepackage{pifont}
\usepackage{calc}

\usepackage[a4paper]{geometry}
\title{\thetitle}

\makeatletter
\def\namedlabel#1#2{\begingroup
  #2%
  \def\@currentlabel{#2}%
  \phantomsection\label{#1}\endgroup
}
\makeatother

\usepackage{mleftright}
\mleftright

\newcommand{\defn}[1]{{\bfseries #1}}
\newcommand{\ceil}[1]{\lceil #1 \rceil}
\newcommand{\floor}[1]{\lfloor #1 \rfloor}

\newcommand{\eps}{\varepsilon}
\newcommand{\cA}{\mathcal{A}}
\newcommand{\cB}{\mathcal{B}}

\newcommand{\sm}{\setminus}
\newcommand{\es}{\varnothing}

\newcommand{\pr}[0]{\mathbb{P}}
\newcommand{\E}[0]{\mathbb{E}}
\newcommand{\whp}[0]{\textbf{whp}}
\newcommand{\Dist}[1]{\mathsf{#1}}
\newcommand{\Bin}{\Dist{Bin}}
\newcommand{\Bernoulli}{\Dist{Bernoulli}}

\DeclareMathOperator{\odd}{odd}

\usepackage{comment}
\author{Lior Gishboliner \and Michael Krivelevich \and Peleg Michaeli}
\AtEndDocument{\bigskip{\footnotesize%
\par

\textsc{Lior Gishboliner} \par
\textsc{Department of Mathematics, ETH, Z\"urich, Switzerland} \par
  \textit{Email:}
  \href{mailto:lior.gishboliner@math.ethz.ch}{\texttt{lior.gishboliner@math.ethz.ch}} \par
When working on this project, the author was supported by ERC starting grant 633509. \par
  \addvspace{\medskipamount}

\par

\textsc{Michael Krivelevich} \par
\textsc{School of Mathematical Sciences, Tel Aviv University,
  Tel Aviv 6997801, Israel} \par
  \textit{Email:}
  \href{mailto:krivelev@tauex.tau.ac.il}{\texttt{krivelev@tauex.tau.ac.il}} \par
  Research supported in part by USA-Israel BSF grant 2018267 and by ISF grant 1261/17.
  \par 
  \addvspace{\medskipamount}

\par

\textsc{Peleg Michaeli} \par
\textsc{School of Mathematical Sciences, Tel Aviv University,
  Tel Aviv 6997801, Israel} \par
  \textit{Email:}
  \href{mailto:peleg.michaeli@math.tau.ac.il}{\texttt{peleg.michaeli@math.tau.ac.il}} \par
  Research supported by ERC starting grant 676970 RANDGEOM and by ISF grants
  1207/15 and 1294/19.\par
  \addvspace{\medskipamount}

}}

\usepackage{subcaption}
\usepackage{tikz}
\usetikzlibrary{calc}

\begin{document}
\maketitle

%%%%%%%%%%%%%%%%%%%%%%%%%%%%%%%%%%%%%%%%%%%%%%%%%%%%%%%%%%%%%%%%%%%%%%%%%%%%%%%
% Abstract
\begin{abstract}
  We prove that a random graph $G(n,p)$, with $p$ above the Hamiltonicity threshold, is typically such that for any $r$-colouring of its edges there exists a Hamilton cycle with at least $(2/(r+ 1)-o(1))n$ edges of the same colour.
  This estimate is asymptotically optimal.
\end{abstract}
%%%%%%%%%%%%%%%%%%%%%%%%%%%%%%%%%%%%%%%%%%%%%%%%%%%%%%%%%%%%%%%%%%%%%%%%%%%%%%%

\section{Introduction}
\label{sec:intro}
Hamiltonicity is one of the most flourishing and well-studied areas of research in the theory of random graphs, boasting a wide array of results over hundreds of papers.
In fact, the question of finding the threshold for containing a Hamilton path has already been posed by Erd\H{o}s and R\'enyi in their seminal paper on random graphs~\cite{ER60}.
Building on the breakthrough work of P\'{o}sa~\cite{Pos76}, which introduced a method now known as \emph{P\'{o}sa's rotation--extension technique}, Koml\'{o}s and Szemer\'{e}di~\cite{KS83} and independently Bollob\'{a}s~\cite{Bol84} proved the fundamental result that the threshold for the appearance of a Hamilton cycle in the binomial random graph $G(n,p)$ is $p = (\log n + \log\log n)/n$.
For a historical overview and a list of papers on this topic we refer the reader to an annotated bibliography by Frieze~\cite{FriBib}.

A central theme in this area is that the appearance of Hamilton cycles is closely tied to the disappearance of vertices of degree at most $1$.
In fact, having minimum degree $2$ is often thought of as the ``bottleneck" for the appearance of Hamilton cycles.
This perspective is made remarkably precise in the hitting time results of Ajtai, Koml\'{o}s and Szemer\'{e}di~\cite{AKS85} and of Bollob\'{a}s~\cite{Bol84} (see also the survey~\cite{Kri16} for a shorter proof, and~\cite{AK20} for yet another quantitative aspect of this phenomenon). 

With the threshold for Hamiltonicity known, it is natural to ask about the typical structure of the set of Hamilton cycles appearing in $G(n,p)$, for $p$ which is just above the Hamiltonicity threshold. 
For such values of $p$, one might expect the number of Hamilton cycles in $G(n,p)$ to be small, and their structure sparse and fragile. It turns out, however, that this is quite far from the truth. In fact, the set of Hamilton cycles of $G(n,p)$ (for $p$ as above) typically possesses a rich and robust structure. Several concrete manifestations of this phenomenon have been demonstrated in prior works. For example, it is known that the number of Hamilton cycles in $G(n,p)$ is --- in some well-defined quantitative sense --- concentrated around its mean~\cite{GK13}; that the set of Hamilton cycles in $G(n,p)$ typically possesses local resilience properties~\cites{LS12,Mon19,NST19,Sud17}; and that random edge-colourings of $G(n,p)$ typically admit Hamilton cycles coloured according to any prescribed pattern~\cites{EFK18,AF19}.

In this paper, we establish yet another natural ``robustness property" of the set of Hamilton cycles in $G(n,p)$ (for any $p$ above the Hamiltonicity threshold).
The precise problem we will be studying is as follows.
For a graph $G$ and an integer $r \geq 2$, let $M(G,r)$ be the largest integer $M$ such that in any $r$-colouring of the edges of $G$, there will be a Hamilton cycle with at least $M$ edges of the same colour (if $G$ is not Hamiltonian, we set $M(G,r)=0$). 
The problem of estimating $M(G,r)$ is somewhat similar to (though slightly different from) multicolour discrepancy problems.
In the general setting of combinatorial discrepancy theory, one is given a hypergraph $H$ and tries to $r$-colour its vertices in such a way that every hyperedge is coloured as evenly as possible, in the sense that the numbers of vertices of a given colour in every hyperedge $e$ deviates from its ``mean", $|e|/r$, by as little as possible.
The discrepancy of $H$ is then defined as the maximal deviation one is guaranteed to have in any colouring. In the special setting we consider here, the vertices of the hypergraph $H$ are the edges of $G$, and the hyperedges of $H$ are the Hamilton cycles in $G$. We note, however, that the problem of estimating $M(G,r)$ differs from its discrepancy variant in that $M(G,r)$ is only concerned with ``one-sided deviations", namely with colours appearing significantly more (and not less) than what is expected.  
It is worth noting that discrepancy-type problems in graphs were studied for various ``target subgraphs", such as cliques~\cite{ES72}, spanning trees~\cites{EFLS95,BCJP20}, Hamilton cycles~\cite{BCJP20} and clique factors~\cite{BCPT21}.  

It is natural to expect that if $G$ contains only few Hamilton cycles, then one can $r$-colour the edges of $G$ in such a way that every Hamilton cycle sees approximately the same number, i.e.\ roughly $n/r$, of edges of each colour. 
Our main result, \cref{thm:ham}, shows that the situation in $G(n,p)$ (for $p$ above the Hamiltonicity threshold) is typically very different: one is always guaranteed to find a Hamilton cycle which contains significantly more than $n/r$ edges of the same colour. As alluded to earlier, this is yet another indication of the rich structure of the set of Hamilton cycles in $G(n,p)$. 

Before stating our main result, let us recall some standard terminology. 
For a positive integer $n$ and a real $p\in[0,1]$, denote by $G(n,p)$ the \defn{binomial random graph}, namely, the probability space of all simple labelled graphs on $n$ given vertices, where each pair of vertices is connected by an edge independently with probability $p$.
We say that an event $\cA$ in our probability space occurs \defn{with high probability} (or \whp{}) if $\pr(\cA)\to 1$ as $n$ goes to infinity.

\begin{theorem}\label{thm:ham}
  Let $r\ge 2$ be an integer and let $p\ge (\log{n}+\log{\log{n}}+\omega(1))/n$.
  Then $G\sim G(n,p)$ is \whp{} such that in any $r$-colouring of its edges there exists a Hamilton cycle with at least $(2/(r+1)-o(1))n$ edges of the same colour.
\end{theorem}

Using similar tools to those used in the proof of \cref{thm:ham}, we sketch a proof for the following analogous result for perfect matchings.

\begin{theorem}\label{thm:pm}
  Let $r\ge 2$ be an integer and let $p\ge (\log{n}+\omega(1))/n$.
  Then, assuming $n$ is even, $G\sim G(n,p)$ is \whp{} such that in any $r$-colouring of its edges there exists a perfect matching with at least $(1/(r+1)-o(1))n$ edges of the same colour.
\end{theorem}

The fraction $1/(r+1)$ in \cref{thm:pm}, and hence also the fraction $2/(r+1)$ in \cref{thm:ham}, is tight.
In fact, in {\em every} $n$-vertex graph $G$ there exists an $r$-colouring in which in every matching, the maximum number of edges of the same colour is at most $n/(r+1)$.
Such a colouring, which to the best of our knowledge first appeared in~\cite{CL75}, can be described as follows.
Partition $V(G)$ into sets $V_1,\ldots,V_r$ such that $|V_i|=n/(r+1)$ for $i=1,\ldots,r-1$ and $|V_r|=2n/(r+1)$.
For $i=1,\ldots,r$ (in increasing order), colour by $i$ every edge touching $V_i$ that has not already been coloured. Namely, for each $1 \leq i \leq r$, all edges contained in $V_i \cup \dots \cup V_r$ and touching $V_i$ are coloured with colour $i$ (see \cref{fig:ham:tight}).
It is easy to see that any monochromatic matching in this colouring is of size at most $n/(r+1)$.
Moreover, observe that any Hamilton cycle contains at most $2n/(r+1)$ edges of a given colour, as otherwise it would also contain a matching of size larger than $n/(r+1)$, hence also $M(G,r) \le 2n/(r+1)$.

\begin{figure*}[t]
  % The goal of this setup is to make it as similar as possible in shape and size to abstract/quotation
  \captionsetup{width=0.879\textwidth,font=small}
  \centering
  \includegraphics[height=1.2in]{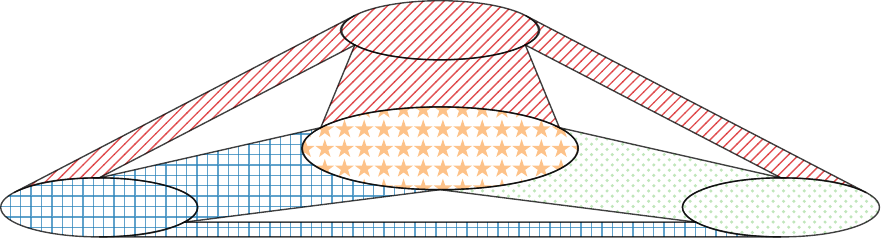}
  \caption{A $4$-coloured complete graph on $n$ vertices.  Each of the small bulbs represents a set of $n/5$ vertices, and the large bulb in the middle represents a set of $2n/5$ vertices.  Any monochromatic matching in this graph is of size at most $n/5$, hence also every Hamilton cycle contains at most $2n/5$ edges of the same colour.}
  \label{fig:ham:tight}
\end{figure*}

The above construction and its analysis suggest a connection between the problem of estimating $M(G,r)$ and the problem of finding monochromatic matchings in $r$-colourings of (the edges of) $G$.
Indeed, our proof of \cref{thm:ham} relies on a new Ramsey-type result for matchings, which may be of independent interest. 

A classical theorem of Cockayne and Lorimer~\cite{CL75} states that for integers $k_1,\dots,k_r \geq 1$ and $n \geq \sum_{i=1}^{r}{(k_i - 1)} + \max\{k_1,\dots,k_r\} + 1$,
every $r$-colouring of the edges of the complete graph $K_n$ contains a monochromatic matching of size $k_i$ in colour $i$ for some $1 \leq i \leq r$.  
The following theorem extends this result to {\em almost complete} host graphs. 

\begin{theorem}\label{thm:match_general}
    Let $r \geq 2$, let $k_1,\dots,k_r \geq 1$, let $0 \leq \delta \leq \frac{1}{2(r+1)}$, let $G$ be a graph with $n$ vertices and at least $(1-\delta)\binom{n}{2}$ edges, and suppose that 
    $\big( 1 - (r+1)\delta \big) n \geq \sum_{i=1}^{r}{(k_i - 1)} + k + 1$, where 
    $k := \max\{k_1,\dots,k_r\}$.
    Then, for every $r$-colouring of the edges of $G$, there is $1 \leq i \leq r$ such that $G$ contains a matching of size $k_i$, all of whose edges are coloured with colour $i$.
\end{theorem}
For $k_1 = \dots = k_r = k$, the condition in Theorem \ref{thm:match_general} becomes $\big( 1 - (r+1)\delta \big) n \geq \linebreak (r+1)(k-1) + 2$, which is satisfied
if $k \leq \big( \frac{1}{r+1} - \delta \big)n$.
Hence, for this case we have the following corollary.
\begin{corollary}
  \label{thm:match}
	Let $r \ge 2$, let $0 \leq \delta \leq \frac{1}{2(r+1)}$, and let $G$ be a graph on $n$ vertices and at least $(1-\delta)\binom{n}{2}$ edges.
  Then, in every $r$-colouring of the edges of $G$ there is a monochromatic matching of size at least
    % $\lfloor (n - 2(r+1)\delta n - 2)/(r+1) \rfloor + 1 \geq 
    % \left( \frac{1}{r+1} - 2\delta \right)n$.
    $\big\lfloor \big( \frac{1}{r+1} - \delta \big)n \big\rfloor$.
\end{corollary}

Our proof of \cref{thm:match_general} is inspired by a new proof of the Cockayne--Lorimer theorem, given in~\cite{XYZ19+}. 

As a second step towards proving \cref{thm:ham}, we will combine \cref{thm:match} with a multicolour version of the sparse regularity lemma (stated here as \cref{thm:sparse_reg}) to prove that in any $r$-colouring of the edges of an $n$-vertex {\em pseudorandom} graph $G$,
there must be a path of length $(2/(r+1)-o(1))n$ in which all but a fixed number of edges are of the same colour.
We will postpone the precise definition of pseudorandomness to  \cref{sec:paths}, and for now only note that as a bi-product, we get the following aesthetically pleasing result.  

\begin{theorem}\label{thm:paths}
  Let $r\ge 2$ be an integer and let $\eps>0$.
  Then there exist $C=C(r,\eps)$ and $K=K(r,\eps)$ such that if $p\ge C/n$, the random graph $G\sim G(n,p)$ is \whp{} such that in any $r$-colouring of its edges there exists a path of length at least $(2/(r+1)-\eps)n$ in which all but at most $K$ of the edges are of the same colour.
\end{theorem}
It is interesting to note that \cref{thm:ham,thm:paths} are nontrivial (and new) even for the extreme case $p=1$, i.e., where the coloured graph is the complete graph.
An immediate corollary of \cref{thm:paths} is that for large enough $C = C(r,\eps)$, \whp{} there is a monochromatic matching of size $(1/(r+1)-\eps)n$ in any $r$-colouring of the edges of $G(n,C/n)$.
In this sense, \cref{thm:paths} is again optimal, as explained before.

% It is interesting to note that \cref{thm:ham,thm:paths} are nontrivial even for the extreme case $p=1$, i.e., where the coloured graph is the complete graph.
A closely related and in fact relatively well studied problem is that of finding long {\em monochromatic} paths in edge-colourings of graphs (which corresponds to requiring $K=0$ in \cref{thm:paths}).
For two colours, this problem was resolved by Gerencs\'{e}r and Gy\'{a}rf\'{a}s in~\cite{GG67} for complete graphs and by Letzter in~\cite{Let16} for random graphs.
For $r \geq 3$ colours, it is conjectured that every edge-colouring of $K_n$ contains a monochromatic path of length $(1/(r-1)-o(1))n$, and that the same holds \whp{} for $G(n,p)$ with $np \to \infty$ (see, e.g.,~\cite{DP17}).
It is known that if true, this would be best possible (even in the complete graph; see, e.g.,~\cite{KS19} and the references therein).
This conjecture was resolved for $r= 3$ by Gy\'{a}rf\'{a}s, Ruszink\'{o}, S\'{a}rk\"{o}zy and Szemer\'{e}di in~\cites{GRSS07,GRSS07*} (for the complete graph) and by Dudek and Pra{\l}at in~\cite{DP17} (for random graphs), and it remains open for all $r \geq 4$.
Accidentally, for $r=2,3$ the two problems --- that of finding a large monochromatic path and that of finding a large path in which all but a constant number of the edges are of the same colour --- have the same answer (both in random and in complete graphs; this follows from \cref{thm:paths} and the aforementioned results of \cites{GG67,Let16,GRSS07,GRSS07*,DP17}).
%The question answered by \cref{thm:paths} was proved (for the complete graph) in~\cite{GG67} and in~\cites{GRSS07,GRSS07*} for $r=2$ and $r=3$ respectively, with $K=0$, i.e., when one looks for a longest monochromatic path.
%This was extended in~\cite{Let16} and in~\cite{DP17} to random graphs with linear edge density, for $r=2$ and $r=3$, respectively, with $K=0$.
For $r\ge 4$, however, these two problems diverge;
% by allowing a fixed number of edges to be coloured differently, \cref{thm:paths} shows that one can find an almost monochromatic path of length $(2/(r+1)-o(1))n$, which is much longer than the longest monochromatic path one is guaranteed to find. 
allowing a fixed number of edges to be coloured differently significantly increases the length of a path one can find, from at most $(1/(r-1)+o(1))n$ for monochromatic paths to $(2/(r+1) - o(1))n$ for {\em almost monochromatic} ones.

A common technique for finding long monochromatic paths, pioneered by Figaj and {\L}uczak in~\cite{FL07} (following an idea by {\L}uczak~\cite{Luc99}), consists of applying the (sparse) regularity lemma and finding large monochromatic {\em connected matchings} in the reduced graph of a regular partition.
In contrast, in order to find an almost monochromatic path, it is sufficient to find a monochromatic (not necessarily connected) matching in the reduced graph.
One can expect --- and we show that this is indeed the case --- that in almost complete graphs (such as the reduced graphs we consider here), one can find substantially larger monochromatic matchings when dropping the requirement that they be connected.
As mentioned above, this is a key step in the proof of \cref{thm:paths}.

%while the longest monochromatic path one can find is known to be of size at most $n/(r-1)$ (even in the complete graph; see, e.g., ~\cite{KS19} and the references therein), by allowing a fixed number of edges to be coloured differently, \cref{thm:paths} shows that one can find a much longer path, of length $(2/(r+1)-o(1))n$ (and again, this is tight).

Let us now say a few words about the remaining ingredients which go into the proof of \cref{thm:ham}.
With \cref{thm:paths} at hand, the proof of \cref{thm:ham} proceeds as follows.
\Cref{thm:paths} gives us a path $P$ of length $(2/(r+1)-o(1))n$ in which all but a fixed number of edges are of the same colour.
Our goal is therefore to extend this path into a Hamilton cycle, or, equivalently, to find a Hamilton path in the remaining set of vertices between neighbours of the endpoints of $P$.
We achieve this by carefully splitting the remaining vertices into two equal sets, each containing many neighbours of the corresponding endpoint of $P$, so that the minimum degree of the graph spanned by each of these sets is at least $2$.
In fact, to do so we need to ``prepare'' our graph, putting aside small degree vertices with their neighbours, and finding $P$ outside this set.
We thus want to find a suitable path not in our random graph but rather in some large induced subgraph thereof; hence we need a generalisation of \cref{thm:paths} to pseudorandom graphs, \cref{thm:paths:pseudo}.
We continue by showing that in each of the two above-mentioned sets there are many Hamilton paths which start at a given point (a neighbour of the corresponding endpoint of $P$), or, more precisely, Hamilton paths with many distinct ends.
The argument relies on the so-called rotation-extension technique, invented by P\'osa in~\cite{Pos76} and has since been applied in numerous papers about Hamiltonicity of random graphs.
We conclude our proof by using expansion properties of our graph to connect the ends of two such Hamilton paths, by that extending $P$ to a Hamilton cycle.

\paragraph{Organisation}
\label{organisation}
We begin by proving \cref{thm:match_general} in \cref{sec:match}.
In \cref{sec:paths} we prove \cref{thm:paths:pseudo}, a generalisation of \cref{thm:paths} to pseudorandom graphs.
At the end of the section we show how to connect the monochromatic linear forest we obtain to a long path, almost all of whose edges are of the same colour.
The goal of \cref{sec:ham:ext} is to introduce fairly general machinery to prepare random graphs in such a way that a path found in some (large) part of the graph can always be extended to a Hamilton cycle.

\paragraph{Notation and terminology}
Let $G=(V,E)$ be a graph.
For two vertex sets $U,W\subseteq V$ we denote by $E_G(U)$ the set of edges of $G$ spanned by $U$ and by $E_G(U,W)$ the set of edges having one endpoint in $U$ and the other in $W$.
The degree of a vertex $v\in V$ is denoted by $d_G(v)$, and we write $d_G(v,U)=|E_G(\{v\},U)|$.
We let $\delta(G)$ and $\Delta(G)$ denote the minimum and maximum degrees of $G$.
When the graph $G$ is clear from the context, we may omit the subscript $G$ in the notations above.

If $f,g$ are functions of $n$ we use the notation $f\sim g$ to denote asymptotic equality, namely, $f\sim g$ if $f=(1+o(1))g$, and we write $f\ll g$ if $f=o(g)$.
For the sake of simplicity and clarity of presentation, we often make no particular effort to optimise the constants obtained in our proofs, and omit floor and ceiling signs when they are not crucial.

\section{Large monochromatic matchings in almost complete graphs}
\label{sec:match}

The goal of this section is to prove \cref{thm:match_general}. 
The primary tool used in the proof is the well-known Tutte--Berge formula (see, e.g.,~\cite{LP}), which we state as follows. For a graph $G$, let $\nu(G)$ denote the maximum size of a matching in $G$, and $\odd(G)$ denote the number of connected components of $G$ whose size is odd.
\begin{theorem}[Tutte--Berge formula]\label{thm:Tutte-Berge}
	Every graph $G$ satisfies 
	\[
  \nu(G) = \frac{1}{2} \cdot |V(G)| - 
	\frac{1}{2} \cdot \max_{U \subseteq V(G)}(\odd(G - U) - |U|).
	\]
\end{theorem}
\noindent
We will also need the following simple lemma. 
\begin{lemma}\label{lem:max_edges}
	Let $G$ be a graph with $n$ vertices and $t$ connected components. Then $|E(G)| \leq \binom{n-t+1}{2}$. 
\end{lemma}

\begin{proof}
	Let $G$ be as in the lemma, and let $C_1,\dots,C_t$ be the connected components of $G$.
  Evidently, $|E(G)| \leq \sum_{i=1}^{t}{\binom{|C_i|}{2}}$.
  Thus, in order to prove the lemma, it suffices to show that the function $g(x_1,\dots,x_t) = \sum_{i=1}^{t}{\binom{x_i}{2}}$ with domain $\{(x_1,\dots,x_t) \in [n]^t \; : \; x_1+\dots+x_n = t\}$ attains its maximum when $x_1 = n-t+1,x_2 = \dots = x_t = 1$, where it equals $\binom{n-t+1}{2}$. So let $(x_1,\dots,x_t) \in [n]^t$ be a maximum point of $g$. It is enough to show that there is (at most) one $1 \leq i \leq t$ such that $x_i \geq 2$. So suppose by contradiction that $x_i,x_j \geq 2$ for some distinct $i,j \in [t]$. Without loss of generality, assume that $x_j \geq x_i$. Now, setting $y_i := x_i - 1$, $y_j := x_j + 1$, and $y_k := x_k$ for $k \in [t] \sm \{i,j\}$, observe that $g(y_1,\dots,y_k) = g(x_1,\dots,x_k) + x_j - (x_i - 1) \geq  g(x_1,\dots,x_k) + 1$, in contradiction to the choice of $x_1,\dots,x_k$. 
\end{proof}

\subsubsection*{Proof of \cref{thm:match_general}}
Let $G$ be a graph with $n$ vertices and at least $(1-\delta)\binom{n}{2}$ edges, and suppose that 
\begin{equation}\label{eq:CL}
\big( 1 - (r+1)\delta \big) n \geq 
\sum_{i=1}^{r}{(k_i - 1)} + k + 1,
\end{equation}
where $k := \max\{k_1,\dots,k_r\}$. We may assume that 
$k \geq \big( \frac{1}{r+1} - \delta \big)n - 1$, because otherwise \eqref{eq:CL} also holds with $k_1$ replaced by $k_1+1$ (which increases $k$ by at most $1$), meaning that we may instead prove the theorem for $k_1+1,k_2,\dots,k_r$ (which evidently implies the statement for $k_1,\dots,k_r$).
% Note that \eqref{eq:CL} in particular means that 
% $\delta \leq \frac{1}{2(r+1)}$.

Fix any $r$-colouring of the edges of $G$.
For each $i \in [r]$, let $G_i$ be the graph on $V(G)$ whose edges are the edges of $G$ which are coloured with colour $i$. 
Our goal is to show that there is $1 \leq i \leq r$ such that $\nu(G_i) \geq k_i$.
So suppose, for the sake of contradiction, that $\nu(G_i) \leq k_i-1$ for every $i\in[r]$.
By \cref{thm:Tutte-Berge}, for each $i \in [r]$ there must be $U_i \subseteq V(G_i) = V(G)$ such that 
\[
\frac{n}{2} - \frac{1}{2} \cdot (\odd(G_i - U_i) - |U_i|) = \nu(G_i) \leq k_i-1,
\]
or, equivalently, 
$\odd(G_i - U_i) \geq n - 2(k_i-1) + |U_i|$. In particular, $G_i - U_i$ has at least $n - 2(k_i-\nolinebreak 1) + |U_i|$ connected components. 
This means that $n - |U_i| = |V(G_i - U_i)| \geq n - 2(k_i-1) + |U_i|$, and hence $|U_i| \leq k_i-1$. By \cref{lem:max_edges}, the following holds for every $i\in[r]$:
\begin{equation*}
    |E(G_i - U_i)|
      \le \binom{(n - |U_i|) - (n - 2(k_i-1) + |U_i|) + 1}{2}
      = \binom{2(k_i-1) - 2|U_i| + 1}{2}.
\end{equation*}
It follows that
\begin{equation}\label{eq:e(G)_upper_bound_Cockayne_Lorimer}
  \begin{split}
    |E(G)|
      &\leq \sum_{i = 1}^{r}{|E(G_i - U_i)|}
        + \#\{e \in E(G) : e \cap (U_1 \cup \dots \cup U_r) \neq \es\} \\
      &\leq \sum_{i = 1}^{r}{\binom{2(k_i-1) - 2|U_i|+1}{2}}
        + \binom{n}{2} - \binom{n-|U_1|-\dots-|U_r|}{2}.
  \end{split} 
\end{equation}
Now, consider the function $g(u_1,\dots,u_r)$ defined by 
\[
g(u_1,\dots,u_r) := 
\sum_{i = 1}^r {\binom{2(k_i-1) - 2u_i+1}{2}} - \binom{n-u_1-\dots-u_r}{2}.
\]
 
\begin{claim}\label{cl:optim}
	Let $u_1,\dots,u_r$ be such that $0 \leq u_i \leq k_i - 1$ for every $i \in [r]$. Then 
  \[
    g(u_1,\dots,u_r) < - \delta \binom{n}{2}.
  \]
\end{claim}
Before proving \cref{cl:optim}, let us complete the proof of \cref{thm:match_general} assuming this claim.
Recall that $|U_i| \leq k_i-1$ for every $i\in[r]$.
Thus, by applying \cref{cl:optim} with $u_i = |U_i|$ ($i\in[r]$), we get that $g(|U_1|,\dots,|U_r|) < -\delta \binom{n}{2}$.
On the other hand, \eqref{eq:e(G)_upper_bound_Cockayne_Lorimer} states that 
$|E(G)| \leq \binom{n}{2} + g(|U_1|,\dots,|U_r|)$, which contradicts our assumption that $|E(G)| \geq (1 - \delta)\binom{n}{2}$.
Thus, in order to complete the proof it suffices to prove \cref{cl:optim}.
\begin{proof}[Proof of \cref{cl:optim}]
	It will be convenient to set $v_i := k_i - 1 - u_i$ for $i\in[r]$. Then $0 \leq v_i \leq k_i - 1$ for every $i \in [r]$.
	Note that the inequality $g(u_1,\dots,u_r) < - \delta \binom{n}{2}$ is equivalent to having 
	\begin{equation}\label{eq:optimization_Cockayne_Lorimer}
		h(v_1,\dots,v_r) := \binom{n - \sum_{i=1}^{r}{(k_i-1)} + \sum_{i=1}^{r}{v_i}}{2} - \sum_{i=1}^{r}{\binom{2v_i+1}{2}} > \delta \binom{n}{2}.
	\end{equation}
	For $1 \leq i \leq r$, observe that if we fix the values of $(v_j : j \in [r] \setminus \{i\})$ and let $v_i$ vary, then the resulting function $h(v_1,\dots,v_r)$ of $v_i$ is a quadratic function in which the coefficient of $v_i^2$ is $-\frac{3}{2} < 0$. Therefore, this function is concave. It follows that for any choice of fixed values of $(v_j : j \in [r] \setminus \{i\})$, the minimum of $h(v_1,\dots,v_r)$ over $0 \leq v_i \leq k_i-1$ is obtained either at $v_i = 0$ or at $v_i = k_i-1$, and is not obtained at any point in the open interval $(0,k_i-1)$. We conclude that if $(v_1,\dots,v_r)$ is a minimum point of $h(v_1,\dots,v_r)$, then $v_i \in \{0,k_i-1\}$ for every $1 \leq i \leq r$. So we see that in order to verify \eqref{eq:optimization_Cockayne_Lorimer}, it is enough to show that $h(v_1,\dots,v_r) > \delta \binom{n}{2}$ for $v_1,\dots,v_r$ satisfying $v_i \in \{0,k_i-1\}$ for every $1 \leq i \leq r$.
	
	% Recall that we set $k := \max\{k_1,\dots,k_r\}$ and assume that 
	% $\big( 1 - 2(r+1)\delta \big) n \geq \linebreak \sum_{i=1}^{r}{(k_i - 1)} + k + 1$. In particular, $\delta \leq \frac{1}{2(r+1)}$.
	Let $I \subseteq [r]$, and suppose that $v_i = k_i - 1$ for $i \in I$ and $v_i = 0$ for $i \in [r] \setminus I$. Then the value of $h(v_1,\dots,v_r)$ is:
	\begin{align*}
	    &\binom{n - \sum_{i \in [r] \setminus I}{(k_i-1)}}{2} - 
	    \sum_{i \in I}{\binom{2k_i-1}{2}}\\ &\geq 
	    \binom{\sum_{i \in I}{(k_i-1)} + k + 1 + (r+1)\delta n}{2} -
	    \sum_{i \in I}{\binom{2k_i-1}{2}} \\ &\geq 
	    \binom{\sum_{i \in I}{(k_i-1)} + k}{2} +  
	    (k+1) \cdot (r+1)\delta n - 
	    \sum_{i \in I}{\binom{2k_i-1}{2}}.
	\end{align*}
	Here, the first inequality uses \eqref{eq:CL}, and the second inequality follows from the fact that $\binom{x+y}{2} \geq \binom{x}{2} + xy$ for all $x,y \geq 0$. Now, since 
	$k \geq \big( \frac{1}{r+1} - \delta \big)n - 1$ (as mentioned in the beginning of the proof) and $\delta \leq \frac{1}{2(r+1)}$ (by assumption), we have 
	$$
	(k+1) \cdot (r+1)\delta n \geq 
	\left( \frac{1}{r+1} - \delta \right)n \cdot (r+1)\delta n \geq 
	\frac{n}{2(r+1)} \cdot (r+1) \delta n = 
	\frac{\delta n^2}{2} > \delta \binom{n}{2}.
	$$
	Thus, to establish \cref{cl:optim}, it suffices to verify that 
	\begin{equation}\label{eq:optimization_Cockayne_Lorimer_aux}
	\binom{\sum_{i \in I}{(k_i-1)} + k}{2} - 
	    \sum_{i \in I}{\binom{2k_i-1}{2}} \geq 0.
	\end{equation}
	Observe that for every $i \in I$, if we fix the values of $(k_j : j \in I \setminus \{i\})$ and consider the left-hand side of \eqref{eq:optimization_Cockayne_Lorimer_aux} as a one-variable function of $k_i$, then this function is quadratic and the coefficient of $k_i^2$ is $-3/2 < 0$. Thus, this function is concave. It follows that at a minimum point of the left-hand side of \eqref{eq:optimization_Cockayne_Lorimer_aux}, we must have $k_i \in \{1,k\}$ for every $i \in I$ (recall that $k_i \leq k$ for every $1 \leq i \leq r$). So let $J \subseteq I$, and suppose that $k_i = k$ for every $i \in J$ and $k_i = 1$ for every $i \in I \setminus J$. Setting $s := |J|$, we see that the left-hand side of \eqref{eq:optimization_Cockayne_Lorimer_aux} equals
	\begin{align*}
	    \binom{s \cdot (k-1) + k}{2} - s \cdot \binom{2k-1}{2} = \frac{(k-1)(s-1)}{2} \cdot \left( (s-1)k - s \right).
	\end{align*}
	So it remains to show that $f(s) := (k-1)(s-1)\cdot \left( (s-1)k - s \right) \geq 0$ for every value of $s$. If $k = 1$ then $f(s) = 0$ for every $s$, so suppose that $k \geq 2$. Now, we have $f(0) = (k-1)k \geq 0$, $f(1) = 0$, and $f(s) \geq (s-1)k-s = (k-1)s - k \geq 2(k-1) - k \geq 0$ for every $s \geq 2$, as required.
\end{proof}
\noindent
With \cref{cl:optim} established, the proof of \cref{thm:match_general} is complete.
\qed

~

It should be noted that a MathOverflow post due to F. Petrov~\cite{Petrov} contains a derivation of the Cockayne--Lorimer result~\cite{CL75} using the Tutte--Berge formula in a similar manner to our proof of \cref{thm:match_general}.

\section{Large monochromatic linear forests in pseudorandom graphs}
\label{sec:paths}
The goal of this section is to prove \cref{thm:paths}.
In fact, we prove a stronger statement, namely \cref{thm:paths:pseudo} below.
This theorem extends \cref{thm:paths} to the more general setting of pseudorandom graphs, and will be used in the proof of \cref{thm:ham}. 

Let us now introduce some definitions. For a pair of disjoint vertex-sets $U,W$ in a graph, the \defn{density} of $(U,W)$ is defined as 
$d(U,W) := |E(U,W)|/(|U||W|)$. For $\gamma,p\in (0,1]$, we say that $G=(V,E)$ is \defn{$(\gamma,p)$-pseudorandom} if 
for any two disjoint $U,W\subseteq V$ with $|U|,|W|\ge\gamma|V|$ we have $|d(U,W) - p| \leq \gamma p$.
We now recall the known fact that if $G = (V,E)$ is $(\gamma,p)$-pseudorandom then every set $U \subseteq V$ of size at least $2\gamma |V|$ satisfies 
\begin{equation}\label{eq:quasirandom}
\left| \frac{|E(U)|}{\binom{|U|}{2}} - p \right| \leq \gamma p.
\end{equation}
To see that \eqref{eq:quasirandom} holds, take a random partition of $U$ into two equal parts $U_1,U_2$ and observe that the expected value of $|E(U_1,U_2)|$ is 
$
|E(U)| \cdot \binom{|U|}{2}^{-1} \cdot |U_1||U_2|.
$
On the other hand, we have $|d(U_1,U_2) - p| \leq \gamma p$ for every such choice of $U_1,U_2$. Therefore, 
\[
\left| \frac{|E(U)|}{\binom{|U|}{2}} - p \right| =   
\left| \frac{\E|E(U_1,U_2)|}{|U_1||U_2|} - p \right| \leq \gamma p,
\]
as required.

Note that if $G$ is a $(\gamma,p)$-pseudorandom graph on $n$ vertices (for any $p \in (0,1]$) then there exists an edge between any two disjoint sets of size at least $\gamma n$.

~

The following is the main result of this section, and will play an important role in the proof of \cref{thm:ham}. 

\begin{theorem}\label{thm:paths:pseudo}
  Let $r\ge 2$ be an integer and let $\eps>0$.
  Then there exist $\gamma=\gamma(r,\eps)$ and $K=K(r,\eps)$ such that the following holds.
  Let $G=(V,E)$ be a $(\gamma,p)$-pseudorandom graph for some $p\in (0,1]$, and suppose $|V|=n$ is large enough (in terms of $r,\eps$).
  Then, in any $r$-colouring of the edges of $G$ there exists a path of length at least $(2/(r+1)-\eps)n$ in which all but at most $K$ of the edges are of the same colour.
\end{theorem}

The proof of \cref{thm:paths:pseudo} relies on (a ``multicolour" version of) the well-known {\em sparse regularity lemma}, proved by Kohayakawa~\cite{Koh97} and R\"{o}dl (see~\cite{Con14}), and later in a stronger form by Scott~\cite{Sco11}. To state this result, we now introduce some additional definitions.
A pair $(U,W)$ of disjoint vertex-sets is called \defn{$(\delta,q)$-regular} if for all $U' \subseteq U$, $W' \subseteq W$ with $|U'| \geq \delta |U|$ and $|W'| \geq \delta |W|$ it holds that $|d(U',W') - d(U,W)| \leq \delta q$. An \defn{equipartition} of a set is a partition in which the sizes of any two parts differ by at most $1$ (to keep the presentation clean, we will ignore divisibility issues and just assume that all parts have the same size). 
Let $G_1,\dots,G_r$ be graphs on the same vertex-set $V$ of size $n$. An equipartition $\{V_1,\dots,V_t\}$ of $V$ is said to be \defn{$(\delta)$-regular} with respect to $(G_1,\dots,G_r)$ if for all but at most $\delta \binom{t}{2}$ of the pairs $(V_i,V_j)$, $1 \leq i < j \leq t$, it holds that for every $\ell\in[r]$, the pair $(V_i,V_j)$ is $(\delta,q)$-regular in $G_{\ell}$, where 
$q := (|E(G_1)| + \dots + |E(G_r)|)/\binom{n}{2}$. 
We are now ready to state the multicolour sparse regularity lemma from \cite{Sco11}. 
\begin{theorem}[Multicolour sparse regularity lemma~\cite{Sco11}]
  \label{thm:sparse_reg}
For every $r,t_0 \geq 1$ and $\delta \in (0,1)$ there exists $T = T(r,t_0,\delta)$ such that for every collection $G_1,\dots,G_r$ of graphs on the same vertex-set $V$, there is an equipartition of $V$ which is $(\delta)$-regular with respect to $(G_1,\dots,G_r)$, and has at least $t_0$ and at most $T$ parts.
\end{theorem}

\noindent
Another tool we will use in the proof of \cref{thm:paths:pseudo} is the following simple lemma from~\cite{BKS12} (see Lemma 4.4 there).

\begin{lemma}
\label{lem:paths_DFS}
Let $n,k \geq 1$ be integers, and let $F$ be a bipartite graph with sides $X,Y$ of size $n$ each. Suppose that there is an edge between every pair of sets $X' \subseteq X$ and $Y' \subseteq Y$ with $|X'| = |Y'| = k$. Then $F$ contains a path of length at least $2n-4k$.
\end{lemma}

The proof of \cref{lem:paths_DFS} proceeds by a careful analysis of the DFS algorithm, an idea which originated in~\cite{BKS12} and has since been widely used in the study of paths in random and pseudorandom graphs (see also~\cite{Kri16} and~\cite{Let16}*{Corollary 2.1}). 

We are now ready to prove \cref{thm:paths:pseudo}. 

\begin{proof}[Proof of \cref{thm:paths:pseudo}]
  Let $r \geq 2$ and let $\eps \in (0,1)$. 
Fix $\delta > 0$ to be small enough so that $\delta < 1/(4r)$ and 
% $\big(1/(r+1) - 2\sqrt{\delta} \big) \cdot (2 - 4\delta) \geq 2/(r+1) - \eps/2$
% (for this second requirement, choosing $\delta \leq \eps^2/300$ should suffice).
$\big(1/(r+1) - 2\delta \big) \cdot (2 - 4\delta) \geq 2/(r+1) - \eps/2$
(for this second requirement, choosing $\delta \leq \eps/12$ suffices).
  Set $t_0 := 1/\delta$, and let $T = T(r,t_0,\delta)$ be as in \cref{thm:sparse_reg}. We will prove the theorem with $\gamma = \gamma(r,\eps) := \eps/(4T)$ and $K = K(r,\eps) := T$. 
  
  Let $p \in (0,1]$ and let $G$ be a $(\gamma,p)$-pseudorandom graph on $n$ vertices (for some sufficiently large $n$). Set $q := |E(G)|/\binom{n}{2}$, and note that by \eqref{eq:quasirandom} we have $(1 - \gamma)p \leq q \leq (1 + \gamma)p$. Let $f : E(G) \rightarrow [r]$ be an $r$-colouring of the edges of $G$.
  For each $i \in [r]$, let $G_i$ be the graph on $V(G)$ whose edges are the edges of $G$ coloured by colour $i$.
  Let $\{V_1,\dots,V_t\}$ be a $(\delta)$-regular equipartition with respect to $(G_1,\dots,G_r)$, where $t_0 \leq t \leq T$. Let $H$ be the graph on $[t]$ in which $\{i,j\} \in E(H)$ if and only if $(V_i,V_j)$ is $(\delta,q)$-regular in $G_{\ell}$ for every $\ell \in [r]$. 
  The definition of a $(\delta)$-regular partition implies that $|E(H)| \geq (1-\delta)\binom{|V(H)|}{2}$. 
  
  We now define a ``reduced" edge-colouring of $H$. Let $\{i,j\}$ be an edge of $H$.
  Since \linebreak $|V_i| = |V_j| = n/t \geq n/T \geq \gamma n$, we have $d_G(V_i,V_j) \geq (1 - \gamma)p \geq p/2$ (as $G$ is $(\gamma,p)$-pseudorandom). 
  Since $d_G(V_i,V_j) = d_{G_1}(V_i,V_j) + \cdots + d_{G_{\ell}}(V_i,V_j)$, there must be some $\ell \in [r]$ such that $d_{G_{\ell}}(V_i,V_j) \geq p/(2r)$. Colour the edge $\{i,j\}$ by colour $\ell$ (if there is more than one possible colour, choose one arbitrarily).
  
  Since $|E(H)| \geq (1-\delta)\binom{|V(H)|}{2}$, \cref{thm:match} implies that $H$ contains a monochromatic matching of size at least
  % $\big(t - 2 - \sqrt{\delta}t \big)/(r+1) \geq 
  % \big( 1/(r+1) - 2\sqrt{\delta} \big)t$, 
  $\big\lfloor \big(1/(r+1) - \delta \big)t \rfloor \geq 
  \big(1/(r+1) - \delta \big)t - 1 \geq 
 \big(1/(r+1) - 2\delta \big)t
  $
  where the inequality holds because $t \geq t_0 = 1/\delta$. Suppose, without loss of generality, that this matching is in colour $1$, and denote its edge-set by $M$. 
  Fix any $e = \{i,j\} \in M$. Since $\{i,j\}$ is an edge of $H$ coloured with colour $1$, it must be the case that $d_{G_1}(V_i,V_j) \geq p/(2r)$ and that $(V_i,V_j)$ is $(\delta,q)$-regular in $G_1$.
  Then for every $V'_i \subseteq V_i,\ V'_j \subseteq V_j$ with $|V'_i| \geq \delta |V_i|$ and $|V'_j| \geq \delta |V_j|$ it holds that $d_{G_1}(V'_i,V'_j) \geq d_{G_1}(V_i,V_j) - \delta q \geq p/(2r) - \delta q \geq p/(2r) - \delta (1 + \gamma)p \geq p/(2r) - \delta \cdot 2p > 0$, where the last inequality holds due to our choice of $\delta$.
  So we see that $G_1$ contains an edge between every pair of sets $V'_i \subseteq V_i,\ V'_j \subseteq V_j$ with $|V'_i| \geq \delta |V_i| = \delta n/t$ and $|V'_j| \geq \delta |V_j| = \delta n/t$.
  By \cref{lem:paths_DFS} with $k := \delta n/t$, the bipartite subgraph of $G_1$ with sides $V_i$ and $V_j$ contains a path $P_e$ of length at least $(2 - 4\delta)n/t$. 
  
  Observe that the paths $(P_e : e \in M)$ are pairwise-disjoint (as $M$ is a matching in $H$), and that the number of vertices covered by these paths is at least 
  \[
  |M| \cdot (2 - 4\delta)n/t \geq 
  \big(1/(r+1) - 2\delta \big)t \cdot (2 - 4\delta)n/t \geq 
  (2/(r+1) - \eps/2)n,
  \]
  where the last inequality uses our choice of $\delta$.
  
  Finally, put $k = |M|$, noting that $k \leq t \leq T$, and enumerate the paths $(P_e : e \in M)$ as $P_1,\dots,P_k$. For each $1 \leq i \leq |M|$, let $A_i,B_i$ denote the first, respectively last, $\gamma n$ vertices of $P_i$. 
  Since $G$ is $(\gamma,p)$-pseudorandom, there exists an edge $e_i=\{b_i,a_{i+1}\}$ between $b_i\in B_i$ and $a_{i+1}\in A_{i+1}$ for every $i=1,\ldots,k-1$.
  Let $a_1$ be the first vertex of $P_1$ and let $b_k$ be the last vertex of $P_k$.
  Let $P$ be the path obtained by concatenating (parts of) the paths $P_1,\ldots,P_k$ using the edges $e_1,\dots,e_{k-1}$, namely,
  \[
    P = a_1 \xrightarrow{P_1} b_1 \xrightarrow{e_1}
        a_2 \xrightarrow{P_2} b_2 \xrightarrow{e_2}
        \cdots \xrightarrow{e_{k-1}}
        a_k \xrightarrow{P_k} b_k.
  \]
   It is easy to see that 
  \[
  |P| \geq |P_1| + \dots + |P_k| - (2k-2)\cdot \gamma n \geq (2/(r+1) - \eps/2)n - 2T\gamma n \geq (2/(r+1)-\eps)n,
  \]
  where in the last inequality we used our choice of $\gamma$. Moreover, all edges of $P$ except for $e_1,\dots,e_{k-1}$ have the same colour. As $k \leq T = K$, the path $P$ satisfies all the required properties, completing the proof.   
\end{proof}

In view of \cref{thm:paths:pseudo}, in order to obtain \cref{thm:paths} it is enough to prove that random graphs (with sufficiently high edge density) are \whp{} pseudorandom.

\begin{lemma}\label{lem:gnp_is_pseudo}
  For every $\gamma>0$ there exists $C=C(\gamma)>0$ such that if $p\ge C/n$ then $G\sim G(n,p)$ is \whp{} $(\gamma,p)$-pseudorandom.
\end{lemma}

In the proof of \cref{lem:gnp_is_pseudo} and in several other proofs in the next section we will make use of the following version of Chernoff bounds (see, e.g., in,~\cite{JLR}*{Chapter 2}).

\begin{theorem}[Chernoff bounds]\label{thm:chernoff}
%  Let $X\sim\Bin(n,p)$ with $\mu=np$ or $X\sim\Hypergeometric(N,K,n)$ with $\mu=nKN^{-1}$, and let $0<\alpha<1<\beta$.
  Let $X=\sum_{i=1}^n X_i$, where $X_i\sim\Bernoulli(p_i)$ are independent, and let $\mu=\E{X}=\sum_{i=1}^n p_i$.
  Let $0<\alpha<1<\beta$.
  Then
  \begin{align*}
    \pr(X\le \alpha\mu) &\le \exp(-\mu(\alpha\log{\alpha}-\alpha+1)),\\
    \pr(X\ge \beta\mu) &\le \exp(-\mu(\beta\log{\beta}-\beta+1)).
  \end{align*}
\end{theorem}

\begin{proof}[Proof of \cref{lem:gnp_is_pseudo}]
  Note that we may assume $\gamma>0$ is arbitrarily small.
  Write $V=V(G)$.
  Fix disjoint $U,W$ with $|U|,|W|\ge\gamma n$ and write $x=|U||W|/n^2\ge\gamma^2$.
  Note that $X:=|E(U,W)|$ is a binomial random variable with $xn^2$ trials and success probability $p$.
  Thus by \cref{thm:chernoff} there exists $c=c(\gamma)>0$ such that
  \[
  \pr(|d(U,W)-p|\geq \gamma p) =
  \pr(|X-pxn^2| \geq \gamma pxn^2)
    \le 2\exp(-cpn^2).
  \]
  Taking $C=C(\gamma)$ to be large enough so that $C > 2/c$, say, we obtain by the union bound that
  \begin{align*}
  &\pr(
    \exists U,W\subseteq V,\ |U|,|W|\ge\gamma n:\ |d(U,W)-p| \geq \gamma p)
    \le 4^n\cdot e^{-2n} = o(1).  \qedhere
  \end{align*}
\end{proof}

With \cref{lem:gnp_is_pseudo}, the proof of \cref{thm:paths} is now complete.

\section{Extending paths to Hamilton cycles}
\label{sec:ham:ext}
The goal of this section is to give a general machinery to ``prepare'' a random graph (above the hamiltonicity threshold) in a way that any path found in some large portion of the graph can be extended, \whp{}, to a Hamilton cycle.
We will then use this machinery to extend the path obtained in \cref{thm:paths:pseudo} to a Hamilton cycle, proving \cref{thm:ham}.
Throughout this section, we assume that $n$ is large enough whenever needed. 
In addition, as the statement in \cref{thm:ham} is clearly monotone in $p$, we will conveniently assume throughout this section that $np-\log{n}-\log\log{n}\ll \log\log{n}$.

\begin{lemma}\label{lem:ham:ext}
  Let $\eps>0$, let $p=(\log{n}+\log{\log{n}}+\omega(1))/n$ and let $G\sim G(n,p)$.
  Then, \whp{}, there exists a partition $V(G)=V^\star\cup V'$ with $|V^\star|\le\eps n$ for which every path $P\subseteq V'$ with $|V(P)|\le 2n/3$ can be extended to a Hamilton cycle in $G$.
\end{lemma}

The proof of \cref{lem:ham:ext} uses P\'osa's rotation--extension technique. 
Let us now recall some corollaries of P\'{o}sa's lemma~\cite{Pos76}.
For an overview of the rotation--extension technique, we refer the reader to~\cites{Kri16}. 
\begin{lemma}[P\'{o}sa's lemma~\cite{Pos76}]\label{lem:Posa}
Let $G$ be a graph, let $P = v_0,\dots,v_t$ be a longest path in $G$, and let $R$ be the set of all $v \in V(P)$ such that there exists a path $P'$ in $G$ with $V(P') = V(P)$ and with endpoints $v_0$ and $v$.
Then $|N(R)| \leq 2|R|-1$.
\end{lemma}

% Recall that a non-edge of $G$ is called a \defn{booster} if adding it to $G$ creates a graph which is either Hamiltonian or whose longest path is longer than that of $G$.
Recall that a non-edge
$\{x,y\}$ of $G$ is called a booster if adding $\{x,y\}$ to $G$ creates a graph which is either Hamiltonian or whose longest path is longer than that of $G$.
For a positive integer $k$ and a positive real $\alpha$ we say that a graph $G=(V,E)$ is a \defn{$(k,\alpha)$-expander} if $|N(U)|\ge\alpha|U|$ for every set $U\subseteq V$ of at most $k$ vertices.
The following is a widely-used fact stating that $(k,2)$-expanders have many boosters. For a proof, see e.g.~\cite{Kri16}.
\begin{lemma}\label{lem:boosters}
  Let $G$ be a connected $(k,2)$-expander which contains no Hamilton cycle.
  Then $G$ has at least $(k+1)^2/2$ boosters. 
\end{lemma}

We now move on to establish some useful properties satisfied \whp{} by $G(n,p)$ (for $p$ as in \cref{lem:ham:Hamiltonian_induced_subgraphs}).

\begin{lemma}\label{lem:gnp:prop}
  Let $\eps>0$ be sufficiently small, let $p=(\log{n}+\log{\log{n}}+\omega(1))/n$, and let $G\sim G(n,p)$.
  Then, \whp{},
  \begin{description}[leftmargin=!,labelwidth=\widthof{\bfseries (P1)}]
    \item[\namedlabel{P:min_max_degree}{(P1)}]
      $\delta(G) \geq 2$ and $\Delta(G) \leq 10\log n$;

    \item[\namedlabel{P:smalldist}{(P2)}]
      No vertex $v \in V(G)$ with $d(v) < \log n/10$ is contained in a $3$- or a $4$-cycle, and every two distinct vertices $u,v \in V(G)$ with $d(u),d(v) < \log n/10$ are at distance at least $5$ apart;
    
    \item[\namedlabel{P:sparse_small_sets}{(P3)}]
      Every set $U \subseteq V(G)$ of size at most $\varepsilon n/100$ spans at most $\varepsilon |U| \log n/10$ edges. 
    
    \item[\namedlabel{P:Us}{(P4)}]
      There exist disjoint sets $U_1,U_2 \subseteq V(G)$ with $|U_1|,|U_2|\le \eps n$ for which the following hold for every $v \in V(G)$:
      \begin{description}[leftmargin=!,labelwidth=\widthof{\bfseries (a)}]
        \item[\namedlabel{P:Us:a}{(a)}]
          If $d(v)\ge \log n/10$ then $d(v,U_1),d(v,U_2)\ge \eps \log n/100$;
        \item[\namedlabel{P:Us:b}{(b)}]
          If $d(v)\le \log n/10$ then $v$ and all of its neighbours are in $U_1$.
      \end{description}
  \end{description} 
\end{lemma}

\begin{proof}[Proof of \ref{P:min_max_degree}]
  For the minimum degree see, e.g.,~\cite{FK}.
  For the maximum degree, since $d(v)\sim\Bin(n-1,p)$ we have
  \[
  \pr(d(v)\ge 10\log{n})
  \le \binom{n}{10\log{n}}p^{10\log{n}}
  \le \left(\frac{enp}{10\log{n}}\right)^{10\log{n}} \ll 1/n,
  \]
  and the statement follows by the union bound.
\end{proof}

\begin{proof}[Proof of \ref{P:smalldist}]
  Write $V=V(G)$ and $\alpha=1/10$.
  Let $1\le\ell\le 4$ and let $P=(v_0,\ldots,v_\ell)$ be a sequence of $\ell+1$ distinct vertices from $V$, where optionally $v_0=v_\ell$.
  Suppose first that $v_0\ne v_\ell$.
  Let $S_0=V\sm\{v_1,v_\ell\}$ and $S_\ell=V\sm\{v_0,v_{\ell-1}\}$.
  Let $\cA_P$ be the event that $P$ is contained in $G$, and for $i=0,\ell$ let $\cB_i$ be the event that $d(v_i,S_i)\le \alpha \log{n}$.
  By \cref{thm:chernoff} we obtain
  that $\pr(\cB_i)\le n^{-0.6}$.
  The events $\cA_P,\cB_0,\cB_\ell$ are mutually independent, hence   $\pr(\cA_P\land \cB_0\land \cB_\ell) \le p^\ell n^{-1.2}$.
  Let $\cA$ be the event that there exists a path $P=v_0,\ldots,v_\ell$ with $\ell\in[4]$ in $G$ such that $\cA_P$ and $d(v_0),d(v_\ell)\le\alpha \log{n}$.
  By the union bound, $\pr(\cA)\le \sum_{\ell=1}^4 n^{\ell+1-1.2}p^\ell=o(1)$.
  The case $v_0=v_\ell$ (which implies $\ell\in\{3,4\}$) is similar.
  Let $S=V\sm \{v_1,v_{\ell-1}\}$ and let $\cB$ be the event $d(v_0,S)\le\alpha\log{n}$.
  As before, $\pr(\cB)\le n^{-0.6}$, and the events $\cA_P,\cB$ are independent, hence $\pr(\cA_p\land \cB)\le p^\ell n^{-0.6}$.
  Let $\cA'$ be the event that there exists a cycle $P$ of length $\ell\in\{3,4\}$ such that $\cA_P$ and $d(v_0)\le\alpha\log{n}$.
  By the union bound, $\pr(\cA')\le \sum_{\ell=3}^4 n^\ell p^\ell n^{-0.6}=o(1)$.
\end{proof}

\begin{proof}[Proof of \ref{P:sparse_small_sets}]
For a given set $U \subseteq V(G)$ and for a given $k \geq 0$, the probability that $|E_G(U)| \geq k$ is at most 
\[
  \binom{\binom{|U|}{2}}{k} \cdot  p^{k} \leq \binom{|U|^2}{k} \cdot p^k \leq \left( \frac{e|U|^2p}{k} \right)^k.
\]
Hence, by the union bound, noting that $p\le 2\log{n}/n$, the probability that \ref{P:sparse_small_sets} does not hold is at most 
\begin{align*}
    \sum_{t = 1}^{\varepsilon n/100}{\binom{n}{t} \cdot 
    \left( \frac{et^2p}{\varepsilon t \log n/10} \right)^{\varepsilon t \log n/10} } 
    &\leq 
    \sum_{t = 1}^{\varepsilon n/100}{ \left( \frac{en}{t} \right)^t \cdot 
    \left( \frac{60t}{\varepsilon n} \right)^{\varepsilon t \log n/10} } \\ &=  
    \sum_{t = 1}^{\varepsilon n/100}{ \left( \frac{60e}{\varepsilon} \cdot \left( \frac{60t}{\varepsilon n} \right)^{\varepsilon \log n/10 - 1} \right)^t } \\ &\leq 
    \sum_{t = 1}^{\varepsilon n/100}{ \left( \frac{60e}{\varepsilon} \cdot 0.6^{\Omega(\varepsilon \log n)} \right)^t } 
    \\ &= 
    \sum_{t = 1}^{\varepsilon n/100}{o(1)^t} = o(1).\qedhere
\end{align*}
%Here, in the first inequality we plugged in $p \leq 2\log n/n$.
\end{proof}

\begin{proof}[Proof of \ref{P:Us}]
  The proof involves an application of the symmetric form of the Local Lemma (see, e.g.,~\cite{AS}*{Chapter~5}; a similar application appears in~\cite{HKS12}).
  Write $V=V(G)$ %and $\alpha=1/10$,
  and let $X=\{v\in V:\ d(v)\le \log{n}/10\}$.
  We start by observing that $X$ is typically small.
  Indeed, by \cref{thm:chernoff} we have $\pr(d(v)\le\log{n}/10)\le n^{-0.6}$, and by Markov's inequality $|X|\le n^{0.5}$ \whp{}.
  By the definition of $X$ we have that $X^+:=X\cup N(X)$ satisfies $|X^+|\le |X| \cdot \log n/10 \leq n^{0.6}$ \nolinebreak \whp{}.
  
  From now on we fix $G$, assuming that $|X^+| \leq n^{0.6}$ and that $G$ satisfies \ref{P:min_max_degree} and \ref{P:smalldist}; these events happen \whp{}.    
  Let $\eps'=1/(\ceil{1/\eps}+1)$ and note that for $\eps<1/2$ we have $\eps/2\le\eps'<\eps$.
  Write $s=1/\eps'$, let $t=\floor{n/s}\sim\eps' n$, and let $A_1,\ldots,A_t,Z$ be a partitioning of the vertices of $G$ into $t$ ``blobs'' $A_i$ of size $s$ and an extra set $Z$ with $|Z|\le s$.
  For $j\in[t]$, let $(x_j^1,x_j^2)$ be a uniformly chosen (ordered) pair of distinct vertices from $A_j$.
  For $i=1,2$ define $U_i'=\{x_j^i\}_{j=1}^t$.
  Clearly, $|U_1'|=|U_2'|=t$ and $U_1'\cap U_2'=\es$.
  For every $v\in V\sm X$, let $\cB_v$ be the event that 
  $d(v,U_i')<\eps'\log{n}/40$ for some $i=1,2$.
  For such $v$, let $L(v)$ be the set of blobs that contain neighbours of $v$,
  namely, $L(v)=\{A_i:\ N(v)\cap A_i\ne\es\}$.
  For $j\in[t]$ write $n_j(v)=|N(v)\cap A_j|$,
  and note that $\sum_j n_j(v) \ge d(v)-s \ge \log{n}/10 - s \geq \log{n}/20$ (for $n$ large enough).
  For $i=1,2$ and $j\in[t]$, let $\chi_j^i(v)$ be the indicator of the event that $x_j^i$ is a neighbour of $v$, and note that $\E\chi_j^i(v)=\eps' n_j(v)$.
  Observe that for $i=1,2$, $d(v,U_i')=\sum_j\chi_j^i(v)$, hence $\E[d(v,U_i')]=\eps' \sum_j n_j(v) \ge \eps'\log{n}/20$.
  Thus, by \cref{thm:chernoff}, $\pr(\cB_v)\le n^{-c}$ for some $c=c(\eps)>0$.

  For two distinct vertices $u,v\in V\sm X$ say that $u,v$ are \defn{related} if $L(u)\cap L(v)\ne\es$.
  For a vertex $u\in V\sm X$, let $R(u)$ be the set of vertices in $V\sm X$ which are related to $u$, and note that $|R(u)|\le s\Delta(G)^2$, which is, by \ref{P:min_max_degree}, at most $C\log^2{n}$ for some $C=C(\eps)>0$.
  Note that $\cB_u$ is mutually independent of the set of events $\{\cB_v\mid v\in (V\sm X)\sm R(u)\}$.
  We now apply the symmetric form of the Local Lemma\footnote{Note that in expectation there are $n^{\Omega(1)}$ vertices $v \in V \setminus X$ for which the event $\cB_v$ occurs. Hence, it is not true that \whp{} every vertex $v \in V \setminus X$ has high degree to both $U'_1$ and $U'_2$. One can then try to fix the situation for the (relatively few) ``unsatisfied" vertices by moving elements into $U'_1$ and $U'_2$ and between these sets. However, moving elements between $U'_1$ and $U'_2$ --- which might be necessary if for example some $v \in V \setminus X$ has all of its neighbours in $U'_1$ (and hence none in $U'_2$) --- can then affect the situation of other vertices. Seeing as the simple union-bound/alterations arguments do not work, we employ the Local Lemma.}: 
  % \footnote{At a first glance, it might seem like one can prove \ref{P:us} using a simple union bound argument, by arguing that \whp{} every vertex $v \in V \setminus X$ will have high degree in both $U'_i$ and $U'_2$. This is, however, false: there are, in expectation, $n^{\Omega(1)}$ vertices $v \in V \setminus X$ for which the event $\B_v$ occurs. One can then try to fix }
  observing that $e n^{-c} \cdot C\log^2{n}<1$ (for large enough $n$), we get that with positive probability, none of the events 
  $(\cB_v : v \in V \setminus X)$ occur, meaning that 
  $d(v,U_i') \geq \eps'\log{n}/40 \geq 
  \varepsilon\log n/80$ for every $v \in V \setminus X$ and $i = 1,2$.  
  We choose $U_1',U_2'$ to satisfy this.
  % We proceed by observing that $X$ is typically small.
  % Indeed, by \cref{thm:chernoff} we have $\pr(d(v)\le\log{n}/10)\le n^{-0.6}$, and by Markov's inequality $|X|\le n^{0.5}$ \whp{}.
  % By the definition of $X$ we have that $X^+:=X\cup N(X)$ satisfies $|X^+|\le |X| \cdot \log n/10 \leq n^{0.6}$ \whp{}.
  Now define $U_1=U_1'\cup X^+$ and $U_2=U_2'\sm X^+$, and note that from the discussion above, $|U_1|,|U_2|\sim \eps' n \le \eps n$.
  Let $v\in V\sm X$.
  The fact that $G$ satisfies \ref{P:smalldist} implies that $v$ has at most $1$ neighbour in $X^+$.
  Thus, for every $v\in V\sm X$ it holds that $d(v,U_1)\ge d(v,U_1')\ge \eps\log{n}/100$ and $d(v,U_2)\ge d(v,U_2')-1\ge \eps\log{n}/100$, as required. 
\end{proof}

In the proof of \cref{lem:ham:ext}, we will argue that \whp{} $G \sim G(n,p)$ is such that every subset $W \subseteq V(G)$ possessing certain properties induces a Hamiltonian graph. To this end, we will use the fact that given such a set $W$ and a relatively sparse expander $H$ on $W$ which is a subgraph of $G$, it is highly likely that there is an edge $e$ of $G$ which is a booster with respect to $H$.
This fact is established in \cref{lem:gnp:props_hitting_boosters} below.
In the proof of \cref{lem:gnp:props_hitting_boosters} we will use the well-known and easy-to-show fact that if a graph $H$ is a $(|V(H)|/4,2)$-expander then $H$ is connected.
Indeed, if (by contradiction) $H$ is not connected, then take a connected component $X$ of size at most $|V(H)|/2$ and a set $U \subseteq X$ of size $\min\{|V(H)|/4,|X|\}$, and observe that $|N(U)| \leq |X| - |U| < 2|U|$, contradicting the assumption that $H$ is a $(|V(H)|/4,2)$-expander. 

\begin{lemma}\label{lem:gnp:props_hitting_boosters}
  Let $c > 0$ be a sufficiently small absolute constant ($c = 10^{-5}$ suffices), let $p=(\log{n}+\log{\log{n}}+\omega(1))/n$ and let $G\sim G(n,p)$. Then, \whp{}, $G$ satisfies the following: for every $W \subseteq V(G)$ of size $|W| \geq 0.1n$ and for every $(|W|/4,2)$-expander $H$ on $W$ which is a subgraph of $G$ and has at most $c n \log n$ edges, if $H$ is not Hamiltonian then $G$ contains a booster with respect to $H$. 
\end{lemma} 
\begin{proof}
  We use a first moment argument. Evidently, the number of choices for the set $W$ is at most $2^n$.
  Let us fix a choice of $W$. For each $t$, the number of choices of $H$ for which $|E(H)| = t$ is at most 
  \[
  \binom{\binom{|W|}{2}}{t} \leq \binom{n^2}{t} \leq 
  \left( \frac{en^2}{t} \right)^t.
  \]
  Now let $H$ be a non-Hamiltonian $(|W|/4,2)$-expander on $W$, and set $t := |E(H)|$.
  As mentioned above, $H$ is connected.
  By \cref{lem:boosters}, $H$ has at least $(|W|/4)^2/2 = |W|^2/32 \geq n^2/3200$ boosters. Now, the probability that $G$ contains $H$ but no booster thereof is at most 
  \[
  p^t \cdot (1 - p)^{n^2/3200} \leq 
  p^t \cdot \left( 1 - \frac{\log n}{n} \right)^{n^2/3200} \leq 
  \left( \frac{2\log n}{n} \right)^t \cdot \exp\left( {-n\log n/3200} \right). 
  \]
  Summing over all choices of $W$ and $H$, we see that the 
  probability that the assertion of the lemma does not hold is at most
  \begin{equation}\label{eq:hitting_expander_boosters_union_bound}
  2^n \cdot \exp\left( {-n\log n/3200} \right) \cdot 
  \sum_{t = 1}^{cn\log n}{\left( \frac{2en\log n}{t} \right)^t}. 
  \end{equation}
  Setting $g(t) := (2en\log n/t)^t$, we note that $g'(t) = g(t) \cdot \left( \log(2en\log n/t) - 1 \right) > 0$ for every $t$ in the range of the sum in \eqref{eq:hitting_expander_boosters_union_bound}, assuming $c<1$, say. Thus, this sum is not larger than
  \[
  c n\log n \cdot 
  \left( 2e/c \right)^{c n\log n} = 
  \exp \left( {\left( \log(2e/c) \cdot c + o(1) \right) n \log n} \right).
  \]
  Now, if $c$ is small enough so that $\log(2e/c) \cdot c < 1/3200$, we get that \eqref{eq:hitting_expander_boosters_union_bound} tends to $0$ as $n$ tends to infinity. This completes the proof. 
\end{proof}

The following lemma states that a graph possessing certain simple properties is necessarily an expander. Statements of this type are fairly common in the study of Hamiltonicity of random graphs (see, e.g.,~\cite{Kri16}). For completeness, we include a proof. 
\begin{lemma}\label{lem:expander_sufficient_conditions}
Let $m,d \geq 1$ be integers and let $H$ be a graph on $h\ge 4m$ vertices satisfying the following properties:
\begin{enumerate}
    \item $\delta(H) \geq 2$;
    \item No vertex $v \in V(H)$ with $d(v) < d$ is contained in a $3$- or a $4$-cycle, and every two distinct vertices $u,v \in V(H)$ with $d(u),d(v) < d$ are at distance at least $5$ apart;
    \item Every set $U \subseteq V(H)$ of size at most $5 m$ contains at most $d|U| / 10$ edges;
    \item There is an edge between every pair of disjoint sets $U_1,U_2 \subseteq V(H)$ of size $m$ each.
\end{enumerate}
Then $H$ is an $(h/4,2)$-expander. 
\end{lemma}
\begin{proof}
  Our goal is to show that for every $U \subseteq V(H)$ with $|U| \le h/4$ it holds that $|N(U)| \geq 2|U|$.
  So let $U \subseteq V(H)$ be such that $|U| \leq h/4$.
  Suppose first that $|U| \geq m$. Since there evidently is no edge between $U$ and $V(H) \sm (U \cup N(U))$, it must be the case that $|V(H) \sm (U \cup N(U)| < m$ by Item~4. So we have $|U \cup N(U)| > h - m$ and hence $|N(U)| > h - m - |U| \geq \frac{3}{4}h - m \geq \frac{1}{2}h \geq 2|U|$, as required. Here we used the assumption that $h \geq 4m$ as well as the fact that $|U| \leq h/4$.
  
  Suppose now that $|U| \leq m$. Let $X$ be the set of all $u \in U$ satisfying $d(u) < d$, and set $Y := U \sm X$. We claim that $|N(Y)| \geq 4|Y|$. Suppose, for the sake of contradiction, that $|N(Y)| < 4|Y|$. Then $|Y \cup N(Y)| < 5|Y| \leq 5|U| \leq 5m$. On the other hand, the definition of $Y$ implies that $H$ has at least $d|Y|/2$ edges incident to vertices of $Y$. Since all of these edges are contained in $Y \cup N(Y)$, we see that $Y \cup N(Y)$ contains at least $d|Y|/2 > d \cdot |Y \cup N(Y)|/10$ edges. But this stands in contradiction with Item~3. Thus, $|N(Y)| \geq 4|Y|$. 
  
  Next, note that by Item~2, every two elements of $X$ are at distance at least $5$; in particular, $X$ is an independent set, and every two elements of $X$ have disjoint neighbourhoods. Now Item~1 implies that $|N(X)| \geq 2|X|$. 
  
  Observe that each vertex of $Y$ has at most one neighbour in $X \cup N(X)$, for otherwise there would be a $4$-cycle containing an element of $X$ or a pair of elements of $X$ at distance at most $4$, both of which are impossible due to Item~2. So we conclude that $|N(Y) \cap (X \cup N(X))| \leq |Y|$, and hence $|N(Y) \sm (X \cup N(X))| \geq |N(Y)| - |Y|$. All in all, we get that
 \begin{align*}
  |N(U)| &= |N(X) \sm Y| + |N(Y) \sm (X \cup N(X))| \geq 
  |N(X)| - |Y| + |N(Y) \sm (X \cup N(X))| 
  \\&\geq 
  2|X| - |Y| + |N(Y)| - |Y| \geq 2|X| + 2|Y| = 2|U|,
 \end{align*}
  as required. 
  \end{proof}

\noindent
The following lemma constitutes the main part of the proof of \cref{lem:ham:ext}.
\begin{lemma}\label{lem:ham:Hamiltonian_induced_subgraphs}
Let $\eps > 0$.
For $p=(\log{n}+\log{\log{n}}+\omega(1))/n$, the random graph $G\sim G(n,p)$ satisfies the following \whp{}.
Let $W\subseteq V(G)$ be such that $|W| \geq 0.1n$, and for every $v\in W$ it holds that 
$d(v,W)\ge \min\{ d(v), \eps \log n\}$. Then
for every $w \in W$ there exists $Y\subseteq W$ with $|Y|\ge n/40$ such that for each $y\in Y$, there is a Hamilton path in $G[W]$ whose endpoints are $w$ and $y$.
\end{lemma}

\begin{proof}
  We may and will assume $\eps$ is sufficiently small (it is enough to have $\eps \leq \min\{1/10,c/2\}$, where $c$ is the constant from \cref{lem:gnp:props_hitting_boosters}).
    We will assume that the events defined in \cref{lem:gnp_is_pseudo}, \cref{lem:gnp:prop} and \cref{lem:gnp:props_hitting_boosters} hold (this happens \whp{}), and show that in this case, the assertion of \cref{lem:ham:Hamiltonian_induced_subgraphs} holds as well. 

    It will be convenient to set $d_0 := \eps \log n$.
    Let $W \subseteq V(G)$ be as in the statement of \cref{lem:ham:Hamiltonian_induced_subgraphs}.
    We select a random spanning subgraph $H$ of $G[W]$ as follows. For each $v \in W$, if $d(v,W) < d_0$ then add to $H$ all edges of $G[W]$ incident to $v$. Otherwise, namely if $d(v,W) \geq d_0$, then randomly select a set of $d_0$ edges of $G[W]$ incident to $v$ and add these to $H$. Note that $|E(H)| \leq |W| \cdot d_0 \leq \eps n \log n$. On the other hand, our assumption that $d(v,W)\ge \min\{ d(v), \eps \log n\}$ for every $v \in W$ implies that 
    $\delta(H) \geq \min\{\delta(G), d_0\}$.
    Hence, as $d_0 \geq 2$ (for large enough $n$), we have $\delta(H) \geq 2$ by Property \ref{P:min_max_degree}
      of \cref{lem:gnp:prop}.
    
    We claim that with positive probability (in fact, \whp), $H$ is a $(|W|/4,2)$-expander. 
    In light of \cref{lem:expander_sufficient_conditions},
    it is sufficient to show that with positive probability, $H$ satisfies Conditions 1--4 in that lemma.
    Here, we will choose the parameters of \cref{lem:expander_sufficient_conditions} as $d := d_0$ and $m := \eps n/500$.
    We already showed that $\delta(H) \geq 2$ (which is Condition 1 in \cref{lem:expander_sufficient_conditions}).
    Condition 2 holds because $H$ is a subgraph of $G$ and because the analogous statement holds for $G$, as guaranteed by Property \ref{P:smalldist} in \cref{lem:gnp:prop} (here we assume that $\eps \leq 1/10$). Similarly, Condition 3 holds because $H$ is a subgraph of $G$ and due to Property \ref{P:sparse_small_sets} in \cref{lem:gnp:prop} (note that $5m = \varepsilon n /100$). 

    Let us now prove that Condition 4 holds. Let 
    $U_1,U_2 \subseteq V(H) = W$ be disjoint sets satisfying $|U_1|,|U_2| = m = \eps n/500$. Since $G$ is $(\gamma,p)$-pseudorandom with $\gamma = \varepsilon/500$ (in fact, with $\gamma = o(1)$, see \cref{lem:gnp_is_pseudo}), \nolinebreak we \nolinebreak have 
    \begin{equation}\label{eq:sparse_expander_expansion}
    |E_G(U_1,U_2)| \geq 
    (1 - \gamma)p \cdot |U_1||U_2| \geq \frac{|U_1||U_2|\log n}{2n} \geq \frac{\eps^2 n \log n}{500000} = \Omega(n\log n)
    \; .
    \end{equation}
    % Now, let us consider the distribution of the number of edges between $U_1$ and $U_2$ in $H$ 
    %(where the randomness is with respect to the choice of $H$).
    Now, let us bound (from above) the probability that $|E_H(U_1,U_2)| = 0$ (where the randomness is with respect to the choice of $H$).
    Recall that $H$ is defined by choosing, for each $v \in W$, a random set $E(v)$ of $\min\{d(v,W),d_0\}$ edges of $G[W]$ incident to $v$, with all choices made uniformly and independently, and letting $E(H) = \bigcup_{v \in W}{E(v)}$. 
    Fix any $u_1 \in U_1$ with $d(u_1,U_2) \geq 1$, and let $\mathcal{A}_{u_1}$ be the event that 
    there is no edge in $E(u_1)$ with an endpoint in $U_2$. 
    %Fix any $u_1 \in U_1$ with $d(u_1,U_2) \geq 1$, and let $\mathcal{A}_{u_1}$ be the event that when choosing the edges of $H$ incident to $u_1$, no edge of the form $\{u_1,u_2\} \in E(G)$ with $u_2 \in U_2$ is chosen. 
    Observe that if 
    $d(u_1,W) < d_0$ then $\pr(\mathcal{A}_{u_1}) = 0$, and otherwise
     \begin{align*}
    \pr(\mathcal{A}_{u_1}) &= \binom{d(u_1,W)-d(u_1,U_2)}{d_0}/\binom{d(u_1,W)}{d_0} = 
    \prod_{i=0}^{d_0-1}{\frac{d(u_1,W) - d(u_1,U_2) - i}{d(u_1,W)-i}} 
    \\&\leq 
    \left(  
    1 - \frac{d(u_1,U_2)}{d(u_1,W)}
    \right)^{d_0} 
    \leq
    \left(  
    1 - \frac{d(u_1,U_2)}{\Delta(G)}
    \right)^{d_0} 
    \leq 
    e^{-d(u_1,U_2) \cdot \frac{d_0}{\Delta(G)}} \leq 
    e^{-\varepsilon d(u_1,U_2)/10} \;.
    \end{align*}
    Here, in the last inequality we used Property \ref{P:min_max_degree} in \cref{lem:gnp:prop}.
    Note that the events \linebreak $(\mathcal{A}_{u_1} : u_1 \in U_1)$ are independent, and that if $E_H(U_1,U_2) = \es$ then $\mathcal{A}_{u_1}$ occurred for every $u_1 \in U_1$ with $d(u_1,U_2) \geq 1$. It now follows that
    \[
    \pr\left( E_H(U_1,U_2) = \es \right) \leq 
    \exp \left( -\frac{\eps}{10} \cdot \sum_{u_1 \in U_1}{d(u_1,U_2)} \right) = 
    \exp \left( -\frac{\eps}{10} \cdot |E_G(U_1,U_2)| \right) \leq 
    e^{-\Omega(n \log n)},
    \]
    where in the last inequality we used \eqref{eq:sparse_expander_expansion}. 
    By taking the union bound over all at most $2^{2n}$ choices of $U_1,U_2$, we see that with high probability, 
    $E_H(U_1,U_2) \neq \es$ for every pair of disjoint sets $U_1,U_2 \subseteq W$ of size $m$ each.
    
    Finally, we apply \cref{lem:expander_sufficient_conditions} to conclude that \whp{} $H$ is a $(|W|/4,2)$-expander.
    From now on, we fix such a choice of $H$.
    Before establishing the assertion of the lemma, we first show that $G[W]$ is Hamiltonian. 
    To find a Hamilton cycle in $G[W]$, we define a sequence of graphs $H_i$, $i \geq 0$, as follows.
    To begin, set $H_0 = H$.
    For each $i \geq 0$, if $H_i$ is Hamiltonian then stop, and otherwise take a booster of $H_i$ contained in $G[W]$ and add it to $H_i$ to obtain $H_{i+1}$.
    That such a booster exists is guaranteed by \cref{lem:gnp:props_hitting_boosters}, as we will always have $|E(H_i)| \leq |E(H)| + |W| \leq |E(H)| + n \leq \eps n\log n + n \leq c/2 \cdot n\log n + n \leq c n \log n$, provided that $\eps$ is smaller than $c/2$, where $c$ is the constant appearing in \cref{lem:gnp:props_hitting_boosters}.
    Note also that $H_i$ is a subgraph of $G[W]$ for each $i \geq 0$. 
    Evidently, this process has to stop (because as long as $H_i$ is not Hamiltonian, the maximum length of a path in $H_i$ is longer than in $H_{i-1}$), thus showing that $G[W]$ must contain a Hamilton cycle, as claimed.
    
    Now let $w \in W$. As $G[W]$ is Hamiltonian, there exists a Hamilton path $P$ of $G[W]$ such that $w$ is one of the endpoints of $P$.
    Evidently, $P$ is a longest path in $G[W]$. Furthermore, note that $G[W]$ is a $(|W|/4,2)$-expander because $H$, a subgraph of $G[W]$, is such an expander. Let $R$ be the set of all $y \in V(P) = W$ such that there exists a Hamilton path $P'$ in $G[W]$ with endpoints $w$ and $y$.
    By \cref{lem:Posa}, we have $|N_{G[W]}(R)| \leq 2|R| - 1$. Now, since $G[W]$ is a $(|W|/4,2)$-expander, it must be the case that $|R| > |W|/4 \geq n/40$. So we see that the assertion of the lemma holds with $Y = R$. This completes the proof.  
\end{proof}

\begin{proof}[Proof of \cref{lem:ham:ext}]
  For convenience we show the existence of a partition $V(G)=V^\star\cup V'$ with $|V^\star|\le 2\eps n$ instead of $|V^\star|\le\eps n$ (this clearly does not matter).
  We assume that $G$ satisfies the properties detailed in \cref{lem:gnp:prop}, and that it is a $(\gamma,p)$-pseudorandom for $\gamma<1/40$ and some $p\in(0,1)$,
  as guaranteed to happen \whp{} by \cref{lem:gnp_is_pseudo}.
  Let $U_1,U_2$ be disjoint subsets of $V=V(G)$ satisfying \ref{P:Us}.
  Set $V^\star= U_1\cup U_2$ and $V'=V\sm V^\star$, and let $P\subseteq V'$ be a path with $|V(P)|\le 2n/3$ and endpoints $a_1,a_2$.
  In particular, $|V^\star|\le 2\eps n$.
  Our goal is to extend $P$ to a Hamilton cycle of $G$.
  Write $V''=V'\sm V(P)$, partition $V''=V''_1\cup V''_2$ as equally as possible.
  For $i=1,2$, let $W_i=V''_i\cup U_i$ and choose a neighbour $w_i$ of $a_i$ in $W_i$; this is possible since $d(a_i,U_i)\ge \eps\log{n}/100$ by \ref{P:Us}.
  Note that $|W_i|\ge n/6$ and for every $v\in W_i$ it holds that $d(v,W_i)\ge \min\{ d(v), \eps\log{n}/100\}$, hence by \cref{lem:ham:Hamiltonian_induced_subgraphs} there exists a set $Y_i\subseteq W_i$ with $|Y_i|\ge n/40$ such that for every $y\in Y_i$ there is a Hamilton path spanning $W_i$ from $w_i$ to $y$.
  Since $G$ is a $(\gamma,p)$-pseudorandom for $\gamma<1/40$, it has an edge $e$ between $Y_1$ and $Y_2$ with endpoints $y_i\in Y_i$, say.
  For $i=1,2$, denote by $Q_{y_i}$ the Hamilton path between $w_i$ and $y_i$.
  We now construct a Hamilton cycle of $G$ as follows (as depicted in \cref{fig:ham:ext}):
  \[
    a_1 \to w_1 \xrightarrow{Q_{y_1}} y_1 \xrightarrow{e} y_2 \xrightarrow{Q_{y_2}} w_2
    \to a_2 \xrightarrow{P} a_1. \qedhere
  \]
\end{proof}

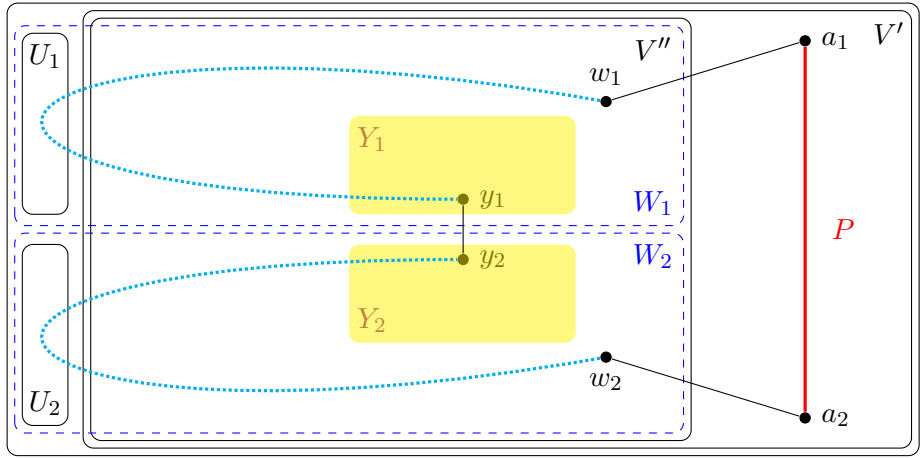
\begin{figure*}[t]
  \captionsetup{width=0.879\textwidth,font=small}
  \centering
  %\definecolor{discred}{HTML}{d7191c}
%\definecolor{discblue}{HTML}{2b83ba}
%\definecolor{discgreen}{HTML}{abdda4}
%\definecolor{discorange}{HTML}{fdae61}

\usetikzlibrary{calc}

\tikzset{
  vertex/.style={fill,circle,inner sep=1.5pt},
}

\begin{tikzpicture}
\def\xa{0cm}
\def\xb{1cm}
\def\xm{6cm}
\def\xc{9cm}
\def\xd{12cm}
\def\xm{6cm}
\def\cC{4.5cm}
\def\xP{10.5cm}

\def\ya{0cm}
\def\yb{3cm}
\def\yc{6cm}
\def\yCa{1.5cm}
\def\yCb{4.5cm}

\def\epl{0.5cm}
\def\eps{0.1cm}

\clip ({\xa-\epl},{\ya-\epl}) rectangle ({\xd+\epl},{\yc+\epl});

% V
\draw[rounded corners] (\xa,\ya) rectangle (\xd,\yc);

% U1
\draw[rounded corners]
  ({\xa+2*\eps},{\yb+2*\eps}) rectangle ({\xb-2*\eps},{\yc-4*\eps});
\node (U1) at ({\xa+5*\eps},{\yc-7*\eps}) {$U_1$};

% U2
\draw[rounded corners]
  ({\xa+2*\eps},{\ya+4*\eps}) rectangle ({\xb-2*\eps},{\yb-2*\eps});
\node (U2) at ({\xa+5*\eps},{\ya+7*\eps}) {$U_2$};

% V'
\draw[rounded corners]
  ({\xb},{\ya+\eps}) rectangle ({\xd-\eps},{\yc-\eps});
\node (V') at ({\xd-4*\eps},{\yc-4*\eps}) {$V'$};

% P
\node[vertex,label={0:$a_1$}] (a1) at (\xP,\yc-\epl) {};
\node[vertex,label={0:$a_2$}] (a2) at (\xP,\ya+\epl) {};
\draw[very thick,red] (a1) -- (a2);
\node[red] (P) at (\xP+\epl,\yb) {$P$};

% V''
\draw[rounded corners]
  ({\xb+\eps},{\ya+2*\eps}) rectangle ({\xc},{\yc-2*\eps});
\node (V'') at ({\xc-5*\eps},{\yc-6*\eps}) {$V''$};

% W1
\draw[rounded corners,blue,dashed]
  ({\xa+\eps},{\yb+0.5*\eps}) rectangle ({\xc-\eps},{\yc-3*\eps});
\node[blue] (W1) at ({\xc-5*\eps},{\yb+3.5*\eps}) {$W_1$};

% W2
\draw[rounded corners,blue,dashed]
  ({\xa+\eps},{\ya+3*\eps}) rectangle ({\xc-\eps},{\yb-0.5*\eps});
\node[blue] (W2) at ({\xc-5*\eps},{\yb-3.5*\eps}) {$W_2$};

% C1, w1
%\draw[thick,purple,dotted] (\cC,\yCb) ellipse ({7*\epl} and {1.5*\epl});
%\node[purple] (C1) at (\cC,\yCb) {$C_1$};
\node[vertex,label={90:$w_1$}] (w1)
  at ($(\cC,\yCb)+(15:{7*\epl} and {1.5*\epl})$) {};
\draw (a1) -- (w1);

% C2, w2
%\draw[thick,purple,dotted] (\cC,\yCa) ellipse ({7*\epl} and {1.5*\epl});
%\node[purple] (C2) at (\cC,\yCa) {$C_2$};
\node[vertex,label={-90:$w_2$}] (w2)
  at ($(\cC,\yCa)+(-15:{7*\epl} and {1.5*\epl})$) {};
\draw (a2) -- (w2);

% x1
%\node[vertex,label={-90:$x_1$}] (x1)
%  at ($(\cC,\yCb)+(-15:{7*\epl} and {1.5*\epl})$) {};
\node (x1)
  at ($(\cC,\yCb)+(-15:{7*\epl} and {1.5*\epl})$) {};

% y1, Y1, P1
\node[vertex,label={0:$y_1$}] (y1) at (\xm,\yb+4*\eps) {};
%\fill[rounded corners,yellow,opacity=0.5]
%  ($(y1)-(2*\eps,2*\eps)$) rectangle ($(x1)+(-4*\eps,2*\eps)$);
\fill[rounded corners,yellow,opacity=0.5]
  ($(y1)-(3*\epl,2*\eps)$) rectangle ($(x1)+(-4*\eps,2*\eps)$);
%\node[brown] (Y1) at (\xm+\eps,\yCb-3*\eps) {$Y_1$};
\node[brown] (Y1) at (\xm+3*\eps-3*\epl,\yCb-3*\eps) {$Y_1$};
%\draw[very thick,cyan,densely dotted] (w1) to[out=165,in=180,looseness=9] (y1);
\draw[very thick,cyan,densely dotted] (w1) to[out=170,in=180,looseness=9.5] (y1);

% x2
%\node[vertex,label={90:$x_2$}] (x2)
%  at ($(\cC,\yCa)+(15:{7*\epl} and {1.5*\epl})$) {};
\node (x2)
  at ($(\cC,\yCa)+(15:{7*\epl} and {1.5*\epl})$) {};

% y2, Y2, P2
\node[vertex,label={0:$y_2$}] (y2) at (\xm,\yb-4*\eps) {};
%\fill[rounded corners,yellow,opacity=0.5]
%  ($(y2)-(2*\eps,-2*\eps)$) rectangle ($(x2)+(-4*\eps,-2*\eps)$);
\fill[rounded corners,yellow,opacity=0.5]
  ($(y2)-(3*\epl,-2*\eps)$) rectangle ($(x2)+(-4*\eps,-2*\eps)$);
%\node[brown] (Y2) at (\xm+\eps,\yCa+3*\eps) {$Y_2$};
\node[brown] (Y2) at (\xm+3*\eps-3*\epl,\yCa+3*\eps) {$Y_2$};
%\draw[very thick,cyan,densely dotted] (w2) to[out=195,in=180,looseness=9] (y2);
\draw[very thick,cyan,densely dotted] (w2) to[out=190,in=180,looseness=9.5] (y2);

\draw (y1) -- (y2);

\end{tikzpicture}
  \caption{Outline of the proof of \cref{lem:ham:ext}.}
  \label{fig:ham:ext}
\end{figure*}

\noindent
We now put together \cref{thm:paths:pseudo} and \cref{lem:ham:ext} in order to prove \cref{thm:ham}. 

\begin{proof}[Proof of \cref{thm:ham}]
  Let $r\ge 2$, $\eps>0$ and $p=(\log{n}+\log{\log{n}}+\omega(1))/n$, let $G\sim G(n,p)$ and consider an $r$-colouring of the edge set of $G$.
  Let $\gamma$ be the constant obtained from \cref{thm:paths:pseudo} by plugging in $r$ and $\eps$.
  Let $V^\star\cup V'$ be the partition guaranteed \whp{} by \cref{lem:ham:ext} which satisfies $n'=|V'|\ge(1-\eps) n$.
  By \cref{lem:gnp_is_pseudo} we know that $G$ is $(\gamma(1-\eps),p)$-pseudorandom (\whp{}), hence $G'=G[V']$ is $(\gamma,p)$-pseudorandom.
  By \cref{thm:paths:pseudo} we know that there exists a path $P$ in $G'$ of length at most $2n'/(r+1)\le 2n/3$ having at least $(2/(r+1)-\eps)n' \ge (2/(r+1)-2\eps)n$ edges of the same colour.
  By \cref{lem:ham:ext} we can, \whp{}, extend $P$ into a Hamilton cycle of $G$, still having at least $(2/(r+1)-2\eps)n$ edges of the same colour.
\end{proof}

\section{Perfect matchings}
We now sketch a proof of \cref{thm:pm}.
The first observation is that with mild modifications of the proof of \cref{lem:ham:ext} we may prove a variant of the following form.
Let $\eps>0$ and $p=(\log{n}+\omega(1))/n$.
Then $G\sim G(n,p)$ \whp{} admits a partition of its vertex set $V(G)=V^\star\cup V'$ with $|V^\star|\le \eps n$ such that {\bf (a)} the set $D_1$ of vertices of degree $1$ in $G$ and its neighbourhood $N(D_1)$ are contained in $V'$;
{\em and} {\bf (b)} for every subset $X$ of $V'$ with $|X|\le 2n/3$ and $D_1\subseteq X$, the subgraph $G[V^\star\cup (V\sm X)]$ contains a Hamilton path.
We omit the proof details.

Having that lemma in hand, we proceed as follows.
Let $M_0$ be the set of edges incident to vertices of $D_1$; there are, \whp{}, $O(\log{n})$ such edges, and they form, \whp{}, a matching.
As $G[V'\sm V(M_0)]$ is (\whp{}) $(\gamma,p)$-pseudorandom by \cref{lem:gnp_is_pseudo}, we know by \cref{thm:paths:pseudo} that it has an almost monochromatic path $P$ of length $(2/(r+1)-\eps)n$, from which we can extract a monochromatic matching of size at least $(1/(r+1)-\eps')n$, for some $\eps'>0$.
Add it to $M_0$, creating an almost monochromatic matching $M_1$ of size at least $(1/(r+1)-\eps')n$.
We now apply the lemma to find a Hamilton path in $G[V^\star\cup (V'\sm V(M_1))]$, from which we extract a matching which completes $M_1$ into a perfect matching, in which at least $(1/(r+1)-\eps')n$ edges are of the same colour. \qed

\acknowledgements{The authors wish to thank the anonymous referees for their careful reading of the paper and useful suggestions which improved its presentation.}

\bibliography{library}

\end{document}